\documentclass[english,a4paper,3p,number,sort&compress,preprint]{elsarticle}
\usepackage{color}
\usepackage{babel}
\usepackage{units}
\usepackage{amsthm}
\usepackage{amsmath}
\usepackage{amssymb}
\usepackage{graphicx}
\usepackage[unicode=true,
 bookmarks=true,bookmarksnumbered=true,bookmarksopen=false,
 breaklinks=true,pdfborder={0 0 0},backref=false,colorlinks=true]
 {hyperref}
\usepackage{breakurl}

\makeatletter
\usepackage{enumitem}		
  \theoremstyle{remark}
  \newtheorem*{notation*}{\protect\notationname}
 \theoremstyle{definition}
 \newtheorem*{defn*}{\protect\definitionname}
  \theoremstyle{plain}
  \newtheorem{thm}{\protect\theoremname}[section]
  \theoremstyle{remark}
  \newtheorem{rem}{\protect\remarkname}[section]
  \theoremstyle{plain}
  \newtheorem{cor}{\protect\corollaryname}[section]
  \theoremstyle{plain}
  \newtheorem{lem}{\protect\lemmaname}[section]
  \theoremstyle{plain}
  \newtheorem*{assumption*}{\protect\assumptionname}

\usepackage[ruled,vlined,titlenotnumbered]{algorithm2e}
\usepackage{eucal}

\makeatother

  \providecommand{\assumptionname}{Assumption}
  \providecommand{\definitionname}{Definition}
  \providecommand{\lemmaname}{Lemma}
  \providecommand{\notationname}{Notation}
  \providecommand{\remarkname}{Remark}
\providecommand{\corollaryname}{Corollary}
\providecommand{\theoremname}{Theorem}

\begin{document}
\global\long\def\vc#1{\mathbf{#1}}
\global\long\def\mx#1{\mathbf{#1}}
\global\long\def\st#1{\mathcal{#1}}
\global\long\def\norm#1{\left\Vert #1\right\Vert }
\global\long\def\herm{^{\mspace{-1mu}\mathsf{H}}}
\global\long\def\cmpl{^{^{\mathsf{c}}}}
\global\long\def\Arg{\text{Arg}}

\title{A Unifying Analysis of Projected Gradient Descent \linebreak{}
for $\ell_{p}$-constrained Least Squares}

\author[ECE]{S.~Bahmani\corref{cor1}}

\ead{sbahmani@cmu.edu}

\author[ECE,LTI]{B.~Raj}

\ead{bhiksha@cs.cmu.edu}

\address[ECE]{Department of Electrical \& Computer Engineering, Carnegie Mellon
University, 5000 Forbes Avenue, Pittsburgh, PA 15213}

\address[LTI]{Language Technologies Institute, Carnegie Mellon University, 5000
Forbes Avenue, Pittsburgh, PA 15213}

\cortext[cor1]{Corresponding author.}
\begin{abstract}
In this paper we study the performance of the Projected Gradient Descent
(PGD) algorithm for $\ell_{p}$-constrained least squares problems
that arise in the framework of Compressed Sensing. Relying on the
Restricted Isometry Property, we provide convergence guarantees for
this algorithm for the entire range of $0\leq p\leq1$, that include
and generalize the existing results for the Iterative Hard Thresholding
algorithm and provide a new accuracy guarantee for the Iterative Soft
Thresholding algorithm as special cases. Our results suggest that
in this group of algorithms, as $p$ increases from zero to one, conditions
required to guarantee accuracy become stricter and robustness to noise
deteriorates.\end{abstract}
\maketitle

\section{Introduction\label{sec:Introduction}}

Least squares problems occur in various signal processing and statistical
inference applications. In these problems the relation between the
vector of noisy observations $\vc y\in\mathbb{C}^{m}$ and the unknown
parameter or signal $\vc x^{\star}\in\mathbb{C}^{n}$ is governed
by a linear equation of the form 
\begin{align}
\vc y & =\mx A\vc x^{\star}+\vc e,\label{eq:Obsvn}
\end{align}
where $\mx A\in\mathbb{C}^{m\times n}$ is a matrix that may model
a linear system or simply contains a set of collected data. The vector
$\vc e\in\mathbb{C}^{m}$ represents the additive observation noise.
Estimating $\vc x^{\star}$ from the observation vector $\vc y$ is
achieved by finding the $\vc x\in\mathbb{C}^{n}$ that minimizes the
squared error $\norm{\mx A\vc x-\vc y}_{2}^{2}$. This least squares
approach, however, is well-posed only if the nullspace of matrix $\mx A$
merely contains the zero vector. The cases in which the nullspace
is greater than the singleton $\left\{ \vc 0\right\} ,$ as in underdetermined
scenarios ($m<n$), are more relevant in a variety of applications.
To enforce unique least squares solutions in these cases, it becomes
necessary to have some prior information about the structure of $\vc x^{\star}$.

One of the structural characteristics that describes parameters and
signals of interest in a wide range of applications from medical imaging
to astronomy is \emph{sparsity}. Since the advent of the theory of
\emph{compressed sensing}, development and analysis of algorithms
that exploit sparsity for estimation in underdetermined problems have
become important topics of study. In the absence of noise $\vc x^{\star}$
can be uniquely determined from the observation vector $\vc y=\mx A\vc x^{\star},$
provided that $\text{spark}\left(\mx A\right)>2\norm{\vc x^{\star}}_{0}$
(i.e., every $2\norm{\vc x^{\star}}_{0}$ columns of $\mx A$ are
linearly independent) \citep{DonohoElad2003optimally}. Then the ideal
estimation procedure could simply be finding the sparsest vector $\vc x$
that incurs no residual error (i.e., $\norm{\mx A\vc x-\vc y}_{2}=0$).
This ideal estimation method can be extended to the case of noisy
observations as well. Formally, given an upper bound $\epsilon$ on
the $\ell_{2}$-norm of the noise, the vector $\vc x^{\star}$ can
be estimated by solving the $\ell_{0}$-minimization 
\begin{align}
\arg\min_{\vc x} & \ \norm{\vc x}_{0}\quad\text{s.t. }\norm{\mx A\vc x-\vc y}_{2}\leq\epsilon,\label{eq:L0min}
\end{align}
where $\norm{\vc x}_{0}$ denotes the $\ell_{0}$-norm\emph{}%
\footnote{\label{fn:quasinorm}The term ``norm'' is used for convenience throughout
the paper. In fact, the $\ell_{0}$ functional violates the positive
scalability property of the norms and the $\ell_{p}$ functionals
with $p\in\left(0,1\right)$ are merely \emph{quasi-norms}.%
} of the vector $\vc x$ that merely counts the number of its non-zero
entries. However, this minimization problem is in general NP-hard
\citep{Natarajan1995sparse}. To avoid the combinatorial computational
cost of \eqref{eq:L0min}, often the $\ell_{0}$-norm is substituted
by the $\ell_{p}$-norm\textsuperscript{\ref{fn:quasinorm}} $\norm{\vc x}_{p}=\left(\sum_{i=1}^{n}\left|x_{i}\right|^{p}\right)^{1/p}$
for some $p\in\left(0,1\right]$ providing the $\ell_{p}$-minimization
\begin{align}
\arg\min_{\vc x} & \ \norm{\vc x}_{p}\quad\text{s.t. }\norm{\mx A\vc x-\vc y}_{2}\leq\epsilon.\label{eq:Lpmin}
\end{align}
 In particular, at $p=1$ the $\ell_{1}$-minimization can be solved
in polynomial time using convex programming algorithms. Several theoretical
and experimental results \citep[see e.g.,][]{Chartrand07,SCY08,Saab201030}
suggest that $\ell_{p}$-minimization with $p\in\left(0,1\right)$
requires fewer observations than the $\ell_{1}$-minimization to produce
accurate estimates. However, $\ell_{p}$-minimization is a non-convex
problem where finding the global minimizer is not guaranteed and can
be computationally more expensive than the $\ell_{1}$-minimization.

An alternative approach in the framework of sparse linear regression
is to solve the sparsity-constrained least squares problem 
\begin{align}
\arg\min_{\vc x} & \ \frac{1}{2}\norm{\mx A\vc x-\vc y}_{2}^{2}\quad\text{s.t. }\norm{\vc x}_{0}\leq s,\label{eq:L0cons}
\end{align}
 where $s=\norm{\vc x^{\star}}_{0}$ is given. Similar to \eqref{eq:L0min}
solving \eqref{eq:L0cons} is not tractable and approximate solvers
must be sought. Several compressed sensing algorithms jointly known
as the \emph{greedy pursuits} including Iterative Hard Thresholding
(IHT) \citep{blumensath_iterative_2009}, Subspace Pursuit (SP) \citep{wei_dai_subspace_2009},
and Compressive Sampling Matching Pursuit (CoSaMP) \citep{NeedellTropp09}
are implicitly approximate solvers of \eqref{eq:L0cons}.

As a relaxation of \eqref{eq:L0cons} one may also consider the $\ell_{p}$-constrained
least squares 
\begin{align}
\arg\min_{\vc x} & \ \frac{1}{2}\norm{\mx A\vc x-\vc y}_{2}^{2}\quad\text{s.t. }\norm{\vc x}_{p}\leq R^{\star},\label{eq:Lpcons}
\end{align}
given $R^{\star}=\norm{\vc x^{\star}}_{p}$. The Least Absolute Shrinkage
and Selection Operator (LASSO) \citep{Tibshirani_96} is a well-known
special case of this optimization problem with $p=1$. The optimization
problem of \eqref{eq:Lpcons} typically does not have a closed-form
solution, but can be (approximately) solved using iterative Projected
Gradient Descent (PGD), which has been outlined in Section \ref{sec:LpPGD}.
Previous studies of these algorithms, henceforth referred to as $\ell_{p}$-PGD,
are limited to the cases of $p=0$ and $p=1$. The algorithm corresponding
to the case of $p=0$ is recognized in the literature as the IHT algorithm.
The Iterative Soft Thresholding (IST) algorithm \citep{beck2009FISTA}
is originally proposed as a solver of the Basis Pursuit Denoising
(BPDN) \citep{ChenDonohoSaunders98}, which is the unconstrained equivalent
of the LASSO with the $\ell_{1}$-norm as the regularization term.
However, the IST algorithm also naturally describes a PGD solver of
\eqref{eq:Lpcons} for $p=1$ \citep[see for e.g,][]{AgarwalPGD-FullLength}
by considering varying shrinkage in iterations, as described in \citep{beck2009FISTA},
to enforce the iterates to have sufficiently small $\ell_{1}$-norm.
The main contribution of this paper is a comprehensive analysis of
the performance of $\ell_{p}$-PGD algorithms for the entire regime
of $p\in\left[0,1\right]$.

In the extreme case of $p=0$ we have the $\ell_{0}$-PGD algorithm
which is indeed the IHT algorithm. Unlike conventional PGD algorithms,
the feasible set ---the set of points that satisfy the optimization
constraints---for IHT is the non-convex set of $s$-sparse vectors.
Therefore, the standard analysis for PGD algorithms with convex feasible
sets that relies on the fact that projection onto convex sets defines
a contraction map will no longer apply. However, imposing extra conditions
on the matrix $\mx A$ can be leveraged to provide convergence guarantees
\citep{blumensath_iterative_2009,RIP-Foucart2010}.

At $p=1$ where \eqref{eq:Lpcons} is a convex program, the corresponding
$\ell_{1}$-PGD algorithm has been studied under the name of IST in
different scenarios (see \citep{beck2009FISTA} and references therein).
Ignoring the sparsity of the vector $\vc x^{\star}$, it can be shown
that the IST algorithm exhibits a sublinear rate of convergence as
a convex optimization algorithm \citep{beck2009FISTA}. In the context
of the \emph{sparse} estimation problems, however, faster rates of
convergence can be guaranteed for IST. For example, in \citep{AgarwalPGD-FullLength}
PGD algorithms are studied in a broad category of regression problems
regularized with ``decomposable'' norms. In this configuration,
which includes sparse linear regression via IST, the PGD algorithms
are shown to possess a linear rate of convergence provided the objective
function---the squared error in our case---satisfies \emph{Restricted
Strong Convexity} (RSC) and \emph{Restricted Smoothness} (RSM) conditions
\citep{AgarwalPGD-FullLength}. These two conditions basically control
the curvature of the objective function being restricted to (nearly)
sparse vectors. Although the results provided in \citep{AgarwalPGD-FullLength}
consolidate the analysis of several interesting problems, they do
not readily extend to the case of $\ell_{p}$-constrained least squares
since the constraint is not defined by a true norm.

In this paper, by considering $\ell_{p}$-balls of given radii as
feasible sets in the general case, we study the $\ell_{p}$-PGD algorithms
that render a continuum of sparse reconstruction algorithms, and encompass
both the IHT and the IST algorithms. In Section \ref{sec:LpPGD} using
the Restricted Isometry Property (RIP) \citep{CandesRombergTao06}
we provide accuracy guarantees for $\ell_{p}$-PGD algorithms which
assert that these algorithms converge to the true signal up to a multiple
of the noise level at a linear rate. Furthermore, our results suggest
that as $p$ increases from zero to one the convergence and robustness
to noise deteriorates. This conclusion is particularly in agreement
with the empirical studies of the \emph{phase} \emph{transition} of
the IST and IHT algorithms provided in \citep{MalekiTST}. Our results
for $\ell_{0}$-PGD coincides with the guarantees for IHT derived
in \citep{RIP-Foucart2010}. Furthermore, to the best of our knowledge
the RIP-based accuracy guarantees we provide for IST, which is the
$\ell_{1}$-PGD algorithm, have not been derived before. The last
section of the paper, Section \ref{sec:discuss}, is dedicated to
discussion of some details and future work.
\begin{notation*}
Throughout the paper we assume that the vectors and matrices have
complex entries unless stated otherwise. The set $\left\{ 1,2,\ldots,n\right\} $
is denoted by $\left[n\right]$ for brevity. We use $\mx M_{\st I}$
to denote restriction of the matrix $\mx M$ to the columns selected
by the set of indices $\st I\subseteq\left[n\right]$. Similarly,
$\vc v|_{\st I}$ denotes restriction of the vector $\vc v$ to the
entries with indices in $\st I$. Depending on the context, the vector
$\vc v|_{\st I}$ may also denote a vector that is equal to the vector
$\vc v$ except for the part supported on $\st I\cmpl$ where it is
zero. The set of non-zero entries (i.e, the support set) and the best
$s$-term approximation of vector $\vc v$ are denoted by $\text{supp}\left(\vc v\right)$
and $\vc v_{s}$, respectively. Furthermore, the matrix $\mx M\herm$
denotes the Hermitian conjugate of the matrix $\mx M$. The inner
product of vectors $\vc u$ and $\vc v$ is denoted by $\left\langle \vc u,\vc v\right\rangle $.
Finally, $\Re\left[\cdot\right]$ and $\Arg\left(\cdot\right)$ denote
the real part and the phase of their arguments, respectively.
\end{notation*}

\section{\label{sec:LpPGD}Projected Gradient Descent for $\ell_{p}$-constrained
Least Squares }

\begin{algorithm}[b]
\DontPrintSemicolon
\SetKwInOut{Input}{input}\SetKwInOut{Output}{output}
\Input{Objective function $f\left(\cdot\right)$ and an operator $\mathrm{P}_\st{Q}\left(\cdot\right)$ that performs projection onto the set $\st{Q}$}
\BlankLine
Choose the initial point $\vc{x}^0\in\st{Q}$\;
$k\longleftarrow 0$\;
\Repeat{halting condition holds}{
Choose a step-size $\eta_k > 0$\;
$\vc{x}^{k+1}\longleftarrow\mathrm{P}_\st{Q}\left(\vc{x}^k-\eta_k\nabla f\left(\vc{x}^k\right)\right)$\;
$k\longleftarrow k+1$\;
}
\Output{the (approximate) minimizer $\vc{x}^{k}$}
\caption{Project Gradient Descent} \label{alg:PGD}
\end{algorithm}

One of the most elementary tools in convex optimization for constrained
minimization is the PGD method. For a differentiable convex objective
function $f\left(\cdot\right)$, a convex set $\st Q$, and a projection
operator $\mathrm{P}_{\st Q}\left(\cdot\right)$ defined by 
\begin{align}
\mathrm{P}_{\st Q}\left(\vc x\right) & =\arg\min_{\vc u}\ \norm{\vc x-\vc u}_{2}^{2}\quad\text{s.t. }\vc u\in\st Q,\label{eq:ProjOp}
\end{align}
 the PGD algorithm solves the minimization 
\begin{align*}
\arg\min_{\vc x} & \ f\left(\vc x\right)\quad\text{s.t. }\vc x\in\st Q
\end{align*}
 via the iterations outlined in Algorithm \ref{alg:PGD}. For example,
in a broad range of applications where the objective function is the
squared error of the form $f\left(\vc x\right)=\frac{1}{2}\norm{\mx A\vc x-\vc y}_{2}^{2}$,
the iterate update equation of the PGD method in Algorithm \ref{alg:PGD}
reduces to 
\begin{align}
\vc x^{k+1} & =\mathrm{P}_{\st Q}\left(\vc x^{k}-\eta_{k}\mx A\herm\left(\mx A\vc x^{k}-\vc y\right)\right).\label{eq:PGD}
\end{align}

In the context of compressed sensing if \eqref{eq:Obsvn} holds and
$\st Q$ is the $\ell_{1}$-ball of radius $\norm{\vc x^{\star}}_{1}$
centered at the origin, Algorithm \ref{alg:PGD} reduces to the IST
algorithm (except perhaps for variable step-size) that solves \eqref{eq:Lpcons}
for $p=1$. By relaxing the convexity restriction imposed on $\st Q$
the PGD iterations also describe the IHT algorithm where $\st Q$
is the set of vectors whose $\ell_{0}$-norm is not greater than $s=\norm{\vc x^{\star}}_{0}$. 

Henceforth, we refer to an $\ell_{p}$-ball centered at the origin
and aligned with the axes simply as an $\ell_{p}$-ball for brevity.
To proceed let us define the set 
\begin{align}
\st F_{p}\left(c\right) & =\left\{ \vc x\in\mathbb{C}^{n}\mid\sum_{i=1}^{n}\left|x_{i}\right|^{p}\leq c\right\} ,\label{eq:FeasibleSet}
\end{align}
 for $c\in\mathbb{R}^{+}$, which describes an $\ell_{p}$-ball. Although
$c$ can be considered as the radius of this $\ell_{p}$-ball with
respect to the metric $d\left(\vc a,\vc b\right)=\norm{\vc a-\vc b}_{p}^{p}$,
we call $c$ the ``$p$-radius'' of the $\ell_{p}$-ball to avoid
confusion with the conventional definition of the radius for an $\ell_{p}$-ball,
i.e., $\max_{\vc x\in\st F_{p}\left(c\right)}\ \norm{\vc x}_{p}$.
Furthermore, at $p=0$ where $\st F_{p}\left(c\right)$ describes
the same ``$\ell_{0}$-ball'' different values of $c$, we choose
the smallest $c$ as the $p$-radius of the $\ell_{p}$-ball for uniqueness.
In this section we will show that to estimate the signal $\vc x^{\star}$
that is either sparse or \emph{compressible} in fact the PGD method
can be applied in a more general framework where the feasible set
is considered to be an $\ell_{p}$-ball of given $p$-radius. Ideally
the $p$-radius of the feasible set should be $\norm{\vc x^{\star}}_{p}^{p}$,
but in practice this information might not be available. In our analysis,
we merely assume that the $p$-radius of the feasible set is not greater
than $\norm{\vc x^{\star}}_{p}^{p}$, i.e., the feasible set does
not contain $\vc x^{\star}$ in its interior.

Note that for the feasible sets $\st Q=\st F_{p}\left(c\right)$ with
$p\in\left(0,1\right]$ the minimum value in \eqref{eq:ProjOp} is
always attained because the objective is continuous and the set $\st Q$
is compact. Therefore, there is at least one minimizer in $\st Q$.
However, for $p<1$ the set $\st Q$ is nonconvex and there might
be multiple projection points in general. For the purpose of the analysis
presented in this paper, however, any such minimizer is acceptable.
Using the axiom of choice, we can assume existence of a choice function
that for every $\vc x$ selects one of the solutions of \eqref{eq:ProjOp}.
This function indeed determines a projection operator which we denote
by $\mathrm{P}_{\st Q}\left(\vc x\right)$. 

Many compressed sensing algorithms such as those of \citep{blumensath_iterative_2009,wei_dai_subspace_2009,NeedellTropp09,Candes2008589}
rely on sufficient conditions expressed in terms of the RIP of the
matrix $\mx A$. We also provide accuracy guarantees of the $\ell_{p}$-PGD
algorithm with the assumption that certain RIP conditions hold. The
following definition states the RIP in its asymmetric form. This definition
is previously \textcolor{black}{proposed in the literature \citep{FoucartLq},}
though in a slightly different format.
\begin{defn*}[RIP]
 Matrix $\mx A$ is said to have RIP of order $s$ with restricted
isometry constants $\alpha_{s}$ and $\beta_{s}$ if they are in order
the smallest and the largest non-negative numbers such that 
\begin{align*}
\beta_{s}\norm{\vc x}_{2}^{2} & \leq\norm{\mx A\vc x}_{2}^{2}\leq\alpha_{s}\norm{\vc x}_{2}^{2}
\end{align*}
holds for all $s$-sparse vectors $\vc x$.
\end{defn*}
In the literature usually the symmetric form of the RIP is considered
in which $\alpha_{s}=1+\delta_{s}$ and $\beta_{s}=1-\delta_{s}$
with $\delta_{s}\in\left[0,1\right]$. For example, in \citep{RIP-Foucart2010}
the $\ell_{1}$-minimization is shown to accurately estimate $\vc x^{\star}$
provided $\delta_{2s}<3/\left(4+\sqrt{6}\right)\approx0.46515$. Similarly,
accuracy of the estimates obtained by IHT, SP, and CoSaMP are guaranteed
provided $\delta_{3s}<\nicefrac{1}{2}$ \citep{RIP-Foucart2010},
$\delta_{3s}<0.205$ \citep{wei_dai_subspace_2009}, and $\delta_{4s}<\sqrt{2/\left(5+\sqrt{73}\right)}\approx0.38427$
\citep{RIP-Foucart2010}, respectively.

As our first contribution, in the following theorem we show that the
$\ell_{p}$-PGD accurately solves $\ell_{p}$-constrained least squares
provided the matrix $\mx A$ satisfies a proper RIP criterion.\textcolor{black}{{}
To proceed we define 
\begin{align*}
\rho_{s} & =\frac{\alpha_{s}-\beta_{s}}{\alpha_{s}+\beta_{s}},
\end{align*}
 which can be interpreted as the equivalent of the standard RIP constant
$\delta_{s}$ in the asymmetric form of RIP.}\textcolor{red}{{} }
\begin{thm}
\label{thm:LpPGD} Let $\vc x^{\star}$ be an $s$-sparse vector whose
compressive measurements are observed according to \eqref{eq:Obsvn}
using a measurement matrix $\mx A$ that satisfies RIP of order $3s$.
To estimate $\vc x^{\star}$ via the $\ell_{p}$-PGD algorithm an
$\ell_{p}$-ball $\widehat{\st B}$ with $p$-radius $\widehat{c}$
(i.e., $\widehat{\st B}=\st F_{p}\left(\widehat{c}\right)$) is given
as the feasible set for the algorithm such that $\widehat{c}=\left(1-\varepsilon\right)^{p}\norm{\vc x^{\star}}_{p}^{p}$
for some%
\footnote{At $p=0$ we have $\left(1-\varepsilon\right)^{0}=1$ which enforces
$\widehat{c}=\norm{\vc x^{\star}}_{0}$. In this case $\varepsilon$
is not unique, but to make a coherent statement we assume that $\varepsilon=0$.%
} $\varepsilon\in\left[0,1\right)$. Furthermore, suppose that the
step-size $\eta_{k}$ of the algorithm can be chosen to obey $\left|\frac{\eta_{k}\left(\alpha_{3s}+\beta_{3s}\right)}{2}-1\right|\leq\tau$
for some $\tau\geq0$. If 
\begin{align}
\left(1+\tau\right)\rho_{3s}+\tau & <\frac{1}{2\left(1+\sqrt{2}\xi\left(p\right)\right)^{2}}\label{eq:mainCondition}
\end{align}

with $\xi\left(p\right)$ denoting the function $\sqrt{p}\left(\frac{2}{2-p}\right)^{1/2-1/p}$,
then $\vc x^{k}$, the $k$-th iterate of the algorithm, obeys 
\begin{align}
\norm{\vc x^{k}-\vc x^{\star}}_{2} & \leq\left(2\gamma\right)^{k}\norm{\vc x^{\star}}_{2}+\frac{2\left(1+\tau\right)}{1-2\gamma}\left(1+\xi\left(p\right)\right)\left(\varepsilon\left(1+\rho_{3s}\right)\norm{\vc x^{\star}}_{2}+\frac{2\sqrt{\alpha_{2s}}}{\alpha_{3s}+\beta_{3s}}\norm{\vc e}_{2}\right)+\varepsilon\norm{\vc x^{\star}}_{2},\label{eq:AccuracyGuarantee}
\end{align}
 where 
\begin{align}
\gamma & =\left(\left(1+\tau\right)\rho_{3s}+\tau\right)\left(1+\sqrt{2}\xi\left(p\right)\right)^{2}.\label{eq:gamma}
\end{align}
\end{thm}
\begin{rem}
Note that the parameter $\varepsilon$ indicates how well the feasible
set $\widehat{\st B}$ approximates the ideal feasible set $\st B^{\star}=\st F_{p}\left(\norm{\vc x^{\star}}_{p}^{p}\right)$.
The terms in \eqref{eq:AccuracyGuarantee} that depend on $\varepsilon$
determine the error caused by the mismatch between $\widehat{\st B}$
and $\st B^{\star}$. Ideally, one has $\varepsilon=0$ and the residual
error becomes merely dependent on the noise level $\norm{\vc e}_{2}$. 
\end{rem}

\begin{rem}
The parameter $\tau$ determines the deviation of the step-size $\eta_{k}$
from $\frac{2}{\alpha_{3s}+\beta_{3s}}$ which might not be known
\emph{a priori}. In this formulation, smaller values of $\tau$ are
desirable since they impose less restrictive condition on $\rho_{3s}$
and also result in smaller residual error. Furthermore, we can naively
choose $\eta_{k}=\norm{\mx A\vc x}_{2}^{2}/\norm{\vc x}_{2}^{2}$
for some $3s$-sparse vector $\vc x\neq\vc 0$ to ensure $1/\alpha_{3s}\leq\eta_{k}\leq1/\beta_{3s}$
and thus $\left|\eta_{k}\frac{\alpha_{3s}+\beta_{3s}}{2}-1\right|\leq\frac{\alpha_{3s}-\beta_{3s}}{2\beta_{3s}}$.
Therefore, we can always assume that $\tau\leq\frac{\alpha_{3s}-\beta_{3s}}{2\beta_{3s}}$.
\end{rem}

\begin{rem}
Note that the function $\xi\left(p\right)$, depicted in Fig. \ref{fig:xi},
controls the variation of the stringency of the condition \eqref{eq:mainCondition}
and the variation of the residual error in \eqref{eq:AccuracyGuarantee}
in terms of $p$. Straightforward algebra shows that $\xi\left(p\right)$
is an increasing function of $p$ with $\xi\left(0\right)=0$. Therefore,
as $p$ increases from zero to one, the RHS of \eqref{eq:mainCondition}
decreases, which implies the measurement matrix must have a smaller
$\rho_{3s}$ to satisfy the sufficient condition \eqref{eq:mainCondition}.
Similarly, as $p$ increases from zero to one the residual error in
\eqref{eq:AccuracyGuarantee} increases. To contrast this result with
the existing guarantees of other iterative algorithms, suppose that
$\tau=0$, $\varepsilon=0$, and we use the symmetric form of RIP
(i.e., $\alpha_{3s}=1+\delta_{3s}$ and $\beta_{3s}=1-\delta_{3s}$)
which implies $\rho_{3s}=\delta_{3s}$. At $p=0$, corresponding to
the IHT algorithm, \eqref{eq:mainCondition} reduces to $\delta_{3s}<\nicefrac{1}{2}$
that is identical to the condition derived in \citep{RIP-Foucart2010}.
Furthermore, the required condition at $p=1$, corresponding to the
IST algorithm, would be $\delta_{3s}<\nicefrac{1}{8}$.
\end{rem}
\begin{figure}
\noindent \centering\includegraphics[width=0.75\textwidth]{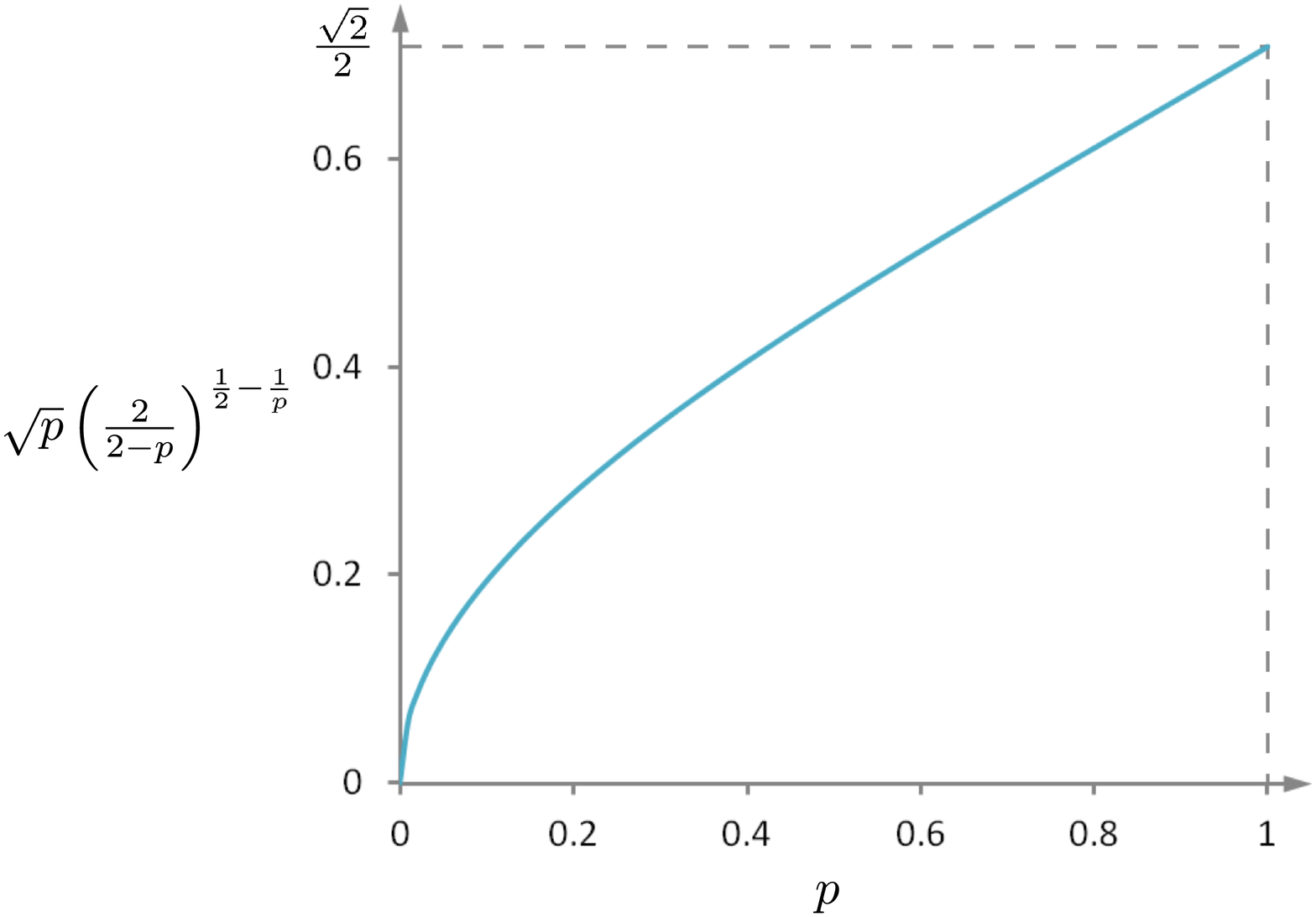}

\caption{\label{fig:xi}Plot of the function $\xi\left(p\right)=\sqrt{p}\left(\frac{2}{2-p}\right)^{\frac{1}{2}-\frac{1}{p}}$
which determines the contraction factor and the residual error.}
\end{figure}

The guarantees stated in Theorem \ref{thm:LpPGD} can be generalized
for nearly sparse or \emph{compressible} signals that can be defined
using power laws as described in \citep{CandesTao06}. The following
corollary provides error bounds for a general choice of $\vc x^{\star}$. 
\begin{cor}
Suppose that $\vc x^{\star}$ is an arbitrary vector in $\mathbb{C}^{n}$
and the conditions of Theorem \ref{thm:LpPGD} hold for $\vc x_{s}^{\star}$,
then the $k$-th iterate of the $\ell_{p}$-PGD algorithm provides
an estimate of $\vc x_{s}^{\star}$ that obeys

\begin{align*}
\norm{\vc x^{k}-\vc x^{\star}}_{2} & \leq\left(2\gamma\right)^{k}\norm{\vc x_{s}^{\star}}_{2}+\frac{2\left(1+\tau\right)\left(1+\xi\left(p\right)\right)}{1-2\gamma}\left(\varepsilon\left(1+\rho_{3s}\right)\norm{\vc x_{s}^{\star}}_{2}+\frac{2\alpha_{2s}}{\alpha_{3s}\!+\!\beta_{3s}}\left(\norm{\vc x^{\star}\!-\!\vc x_{s}^{\star}}_{2}+\frac{\norm{\vc x^{\star}\!-\!\vc x_{s}^{\star}}_{1}}{\sqrt{2s}}\right)\right.\\
 & \left.+\frac{2\sqrt{\alpha_{2s}}}{\alpha_{3s}\!+\!\beta_{3s}}\norm{\vc e}_{2}\right)+\varepsilon\norm{\vc x_{s}^{\star}}_{2}+\norm{\vc x^{\star}-\vc x_{s}^{\star}}_{2}.
\end{align*}
\end{cor}
\begin{proof}
Let $\widetilde{\vc e}=\mx A\left(\vc x^{\star}-\vc x_{s}^{\star}\right)+\vc e$.
We can write $\vc y=\mx A\vc x^{\star}+\vc e=\mx{\mx A}\vc x_{s}^{\star}+\widetilde{\vc e}$.
Thus, we can apply Theorem \ref{thm:LpPGD} considering $\vc x_{s}^{\star}$
as the signal of interest and $\widetilde{\vc e}$ as the noise vector
and obtain

\begin{align}
\norm{\vc x^{k}-\vc x_{s}^{\star}}_{2} & \leq\left(2\gamma\right)^{k}\norm{\vc x_{s}^{\star}}_{2}+\frac{2\left(1+\tau\right)}{1-2\gamma}\left(1+\xi\left(p\right)\right)\left(\varepsilon\left(1+\rho_{3s}\right)\norm{\vc x_{s}^{\star}}_{2}+\frac{2\sqrt{\alpha_{2s}}}{\alpha_{3s}+\beta_{3s}}\norm{\widetilde{\vc e}}_{2}\right)+\varepsilon\norm{\vc x_{s}^{\star}}_{2}.\label{eq:LpPGDE1}
\end{align}
 Furthermore, we have 
\begin{align*}
\norm{\widetilde{\vc e}}_{2} & =\norm{\mx A\left(\vc x^{\star}-\vc x_{s}^{\star}\right)+\vc e}_{2}\\
 & \leq\norm{\mx A\left(\vc x^{\star}-\vc x_{s}^{\star}\right)}_{2}+\norm{\vc e}_{2}.
\end{align*}
Then applying Proposition 3.5 of \citep{NeedellTropp09} yields 
\begin{align*}
\norm{\widetilde{\vc e}}_{2} & \leq\sqrt{\alpha_{2s}}\left(\norm{\vc x^{\star}-\vc x_{s}^{\star}}_{2}+\frac{1}{\sqrt{2s}}\norm{\vc x^{\star}-\vc x_{s}^{\star}}_{1}\right)+\norm{\vc e}_{2}.
\end{align*}
Applying this inequality in \eqref{eq:LpPGDE1} followed by the triangle
inequality $\norm{\vc x^{k}-\vc x^{\star}}_{2}\leq\norm{\vc x^{k}-\vc x_{s}^{\star}}_{2}+\norm{\vc x^{\star}-\vc x_{s}^{\star}}_{2}$
yields the desired inequality. 
\end{proof}
To prove Theorem \ref{thm:LpPGD} first a series of lemmas should
be established. In what follows, $\vc x_{\perp}^{\star}$is a projection
of the $s$-sparse vector $\vc x^{\star}$ onto $\widehat{\st B}$
and $\vc x^{\star}-\vc x_{\perp}^{\star}$ is denoted by $\vc d^{\star}$.
Furthermore, for $k=0,1,2,\ldots$ we denote $\vc x^{k}-\vc x_{\perp}^{\star}$
by $\vc d^{k}$ for compactness.
\begin{lem}
\label{lem:OptimCondition} If $\vc x^{k}$ denotes the estimate in
the $k$-th iteration of $\ell_{p}$-PGD, then 
\begin{align*}
\norm{\vc d^{k+1}}_{2}^{2} & \leq2\Re\left[\left\langle \vc d^{k},\vc d^{k+1}\right\rangle -\eta_{k}\left\langle \mx A\vc d^{k},\mx A\vc d^{k+1}\right\rangle \right]+2\eta_{k}\Re\left\langle \mx A\vc d^{k+1},\mx A\vc d^{\star}+\vc e\right\rangle .
\end{align*}
\end{lem}
\begin{proof}
Note that $\vc x^{k+1}$ is a projection of $\vc x^{k}-\eta_{k}\mx A\herm\left(\mx A\vc x^{k}-\vc y\right)$
onto $\widehat{\st B}$. Since $\vc x_{\perp}^{\star}$ is also a
feasible point (i.e., $\vc x_{\perp}^{\star}\in\widehat{\st B}$)
we have 
\begin{align*}
\norm{\vc x^{k+1}-\vc x^{k}+\eta_{k}\mx A\herm\left(\mx A\vc x^{k}-\vc y\right)}_{2}^{2} & \leq\norm{\vc x_{\perp}^{\star}-\vc x^{k}+\eta_{k}\mx A\herm\left(\mx A\vc x^{k}-\vc y\right)}_{2}^{2}.
\end{align*}
Using \eqref{eq:Obsvn} we obtain 
\begin{align*}
\norm{\vc d^{k+1}-\vc d^{k}+\eta_{k}\mx A\herm\left(\mx A\left(\vc d^{k}-\vc d^{\star}\right)-\vc e\right)}_{2}^{2} & \leq\norm{-\vc d^{k}+\eta_{k}\mx A\herm\left(\mx A\left(\vc d^{k}-\vc d^{\star}\right)-\vc e\right)}_{2}^{2}.
\end{align*}
 Therefore, we obtain 
\begin{align*}
\Re\left\langle \vc d^{k+1},\vc d^{k+1}-2\vc d^{k}+2\eta_{k}\mx A\herm\left(\mx A\vc d^{k}-\left(\mx A\vc d^{\star}+\vc e\right)\right)\right\rangle  & \leq0
\end{align*}
 that yields the the desired result after straightforward algebraic
manipulations.
\end{proof}
The following lemma is a special case of the generalized shifting
inequality proposed in \citep[Theorem 2]{RIP-Foucart2010}. Please
refer to the reference for the proof.
\begin{lem}[\emph{Shifting Inequality} \citep{RIP-Foucart2010}]
\label{lem:ShiftIneq}If $0<p<2$ and 
\begin{align*}
u_{1}\geq u_{2}\geq\cdots\geq u_{l}\ge u_{l+1}\geq\cdots\geq u_{r}\geq u_{r+1}\geq\cdots\geq u_{r+l}\geq0,
\end{align*}
 then for $C=\max\left\{ r^{\frac{1}{2}-\frac{1}{p}},\sqrt{\frac{p}{2}}\left(\frac{2}{2-p}l\right)^{\frac{1}{2}-\frac{1}{p}}\right\} $,
\begin{align}
\left(\sum_{i=l+1}^{l+r}u_{i}^{2}\right)^{\frac{1}{2}} & \leq C\left(\sum_{i=1}^{r}u_{i}^{p}\right)^{\frac{1}{p}}.\label{eq:SIE0}
\end{align}

\end{lem}

\begin{lem}
\label{lem:SupportSub}For $\vc x_{\perp}^{\star}$, a projection
of $\vc x^{\star}$ onto $\widehat{\st B}$, we have $\mathrm{supp}\left(\vc x_{\perp}^{\star}\right)\subseteq\st S=\mathrm{supp}\left(\vc x^{\star}\right)$.\end{lem}
\begin{proof}
Proof is by contradiction. Suppose that there exists a coordinate
$i$ such that $x_{i}^{\star}=0$ but $x_{\perp i}^{\star}\neq0$.
Then one can construct vector $\vc x'$ which is equal to $\vc x_{\perp}^{\star}$
except at the $i$-th coordinate where it is zero. Obviously $\vc x'$
is feasible because $\norm{\vc x'}_{p}^{p}<\norm{\vc x_{\perp}^{\star}}_{p}^{p}\leq\widehat{c}$.
Furthermore, 
\begin{align*}
\norm{\vc x^{\star}-\vc x'}_{2}^{2} & =\sum_{j=1}^{n}\left|x_{j}^{\star}-x'_{j}\right|^{2}\\
 & =\sum_{\substack{j=1\\
j\neq i
}
}^{n}\left|x_{j}^{\star}-x_{_{\perp}j}^{\star}\right|^{2}\\
 & <\sum_{j=1}^{n}\left|x_{j}^{\star}-x_{_{\perp}j}^{\star}\right|^{2}\\
 & =\norm{\vc x^{\star}-\vc x_{\perp}^{\star}}_{2}^{2}.
\end{align*}
 This is a contradiction since by definition 
\begin{align*}
\vc x_{\perp}^{\star} & \in\arg\min_{\vc x}\ \frac{1}{2}\norm{\vc x^{\star}-\vc x}_{2}^{2}\quad\text{s.t. }\norm{\vc x}_{p}^{p}\leq\widehat{c}.
\end{align*}

\end{proof}
To continue, we introduce the following sets which partition the coordinates
of vector $\vc d^{k}$ for $k=0,1,2,\ldots$. As defined previously
in Lemma \ref{lem:SupportSub}, let $\st S=\text{supp}\left(\vc x^{\star}\right)$.
Lemma \ref{lem:SupportSub} shows that $\mbox{supp}\left(\vc x_{\perp}^{\star}\right)\subseteq\st S$,
thus we can assume that $\vc x_{\perp}^{\star}$ is $s$-sparse. Let
$\st S_{k,1}$ be the support of the $s$ largest entries of $\vc d^{k}|_{\st S\cmpl}$
in magnitude, and define $\st T_{k}=\st S\cup\st S_{k,1}$. Furthermore,
let $\st S_{k,2}$ be the support of the $s$ largest entries of $\vc d^{k}|_{\st T_{k}\cmpl}$,
$\st S_{k,3}$ be the support of the next $s$ largest entries of
$\vc d^{k}|_{\st T_{k}\cmpl}$, and so on. We also set $\st T_{k,j}=\st S{}_{k,j}\cup\st S{}_{k,j+1}$
for $j\geq1$. This partitioning of the vector $\vc d^{k}$ is illustrated
in Fig. \ref{fig:ErrorPartition}.

\begin{figure}[t]
\noindent \centering\includegraphics[width=0.7\textwidth]{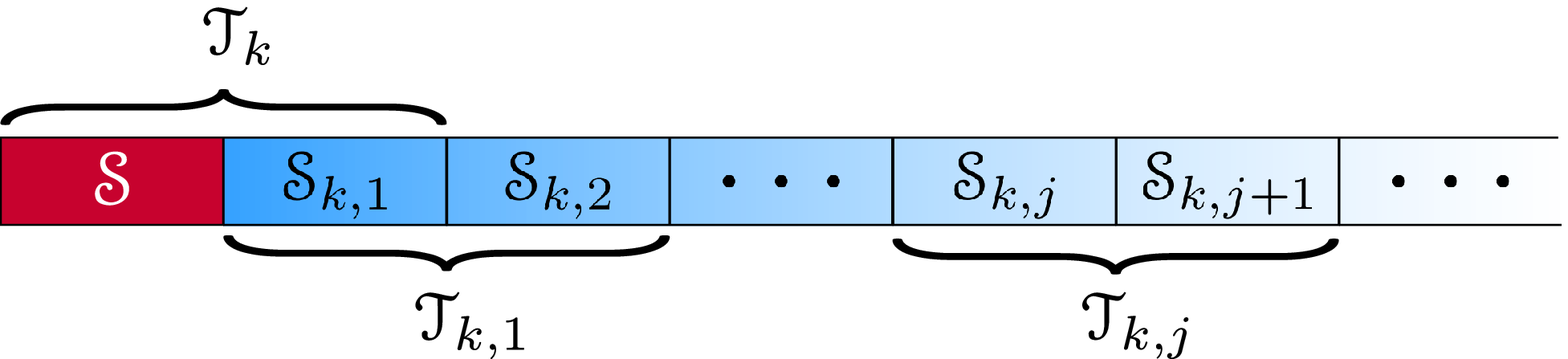}\caption{\label{fig:ErrorPartition}Partitioning of vector $\vc d^{k}=\vc x^{k}-\vc x_{\perp}^{\star}$.
The color gradient represents decrease of the magnitudes of the corresponding
coordinates.}
\end{figure}

\begin{lem}
\label{lem:ell2ellp} For $k=0,1,2,\ldots$ the vector $\vc d^{k}$
obeys
\begin{align*}
\sum_{i\geq2}\norm{\vc d^{k}|_{\st S_{k,i}}}_{2} & \leq\sqrt{2p}\left(\frac{2s}{2-p}\right)^{\frac{1}{2}-\frac{1}{p}}\norm{\vc d^{k}|_{\st S\cmpl}}_{p}.
\end{align*}
\end{lem}
\begin{proof}
Since $\st S_{k,j}$ and $\st S_{k,j+1}$ are disjoint and $\st T_{k,j}=\st S{}_{k,j}\cup\st S{}_{k,j+1}$
for $j\geq1$, we have 
\begin{align*}
\norm{\vc d^{k}|_{\st S_{k,j}}}_{2}+\norm{\vc d^{k}|_{\st S_{k,j+1}}}_{2} & \leq\sqrt{2}\norm{\vc d^{k}|_{\st T_{k,j}}}_{2}.
\end{align*}
Adding over even $j$'s then we deduce 
\begin{align*}
\sum_{j\geq2}\norm{\vc d^{k}|_{\st S_{k,j}}}_{2} & \leq\sqrt{2}\sum_{i\geq1}\norm{\vc d^{k}|_{\st T_{k,2i}}}_{2}.
\end{align*}
Because of the structure of the sets $\st T_{k,j}$, Lemma \ref{lem:ShiftIneq}
can be applied to obtain 
\begin{align}
\norm{\vc d^{k}|_{\st T_{k,j}}}_{2} & \leq\sqrt{p}\left(\frac{2s}{2-p}\right)^{\frac{1}{2}-\frac{1}{p}}\norm{\vc d^{k}|_{\st T_{k,j-1}}}_{p}.\label{eq:ell2ellpE1}
\end{align}
 To be precise, based on Lemma \ref{lem:ShiftIneq} the coefficient
on the RHS should be $C\!=\!\max\left\{ \left(2s\right)^{\frac{1}{2}-\frac{1}{p}},\sqrt{\frac{p}{2}}\left(\!\frac{2s}{2-p}\!\right)^{\frac{1}{2}-\frac{1}{p}}\right\} .$
For simplicity, however, we use the upper bound $C\leq\sqrt{p}\left(\frac{2s}{2-p}\right)^{\frac{1}{2}-\frac{1}{p}}$.
To verify this upper bound it suffices to show that $\left(2s\right)^{\frac{1}{2}-\frac{1}{p}}\leq\sqrt{p}\left(\frac{2s}{2-p}\right)^{\frac{1}{2}-\frac{1}{p}}$
or equivalently $\phi\left(p\right)=p\log p+\left(2-p\right)\log\left(2-p\right)\geq0$
for $p\in\left(0,1\right]$. Since $\phi\left(\cdot\right)$ is a
deceasing function over $\left(0,1\right]$, it attains its minimum
at $p=1$ which means that $\phi(p)\geq\phi(1)=0$ as desired.

Then \eqref{eq:ell2ellpE1} yields 
\begin{align*}
\sum_{j\geq2}\norm{\vc d^{k}|_{\st S_{k,j}}}_{2} & \leq\sqrt{2p}\left(\frac{2s}{2-p}\right)^{\frac{1}{2}-\frac{1}{p}}\sum_{i\geq1}\norm{\vc d^{k}|_{\st T_{k,2i-1}}}_{p}.
\end{align*}
Since $\omega_{1}+\omega_{2}+\cdots+\omega_{l}\leq\left(\omega_{1}^{p}+\omega_{2}^{p}+\cdots+\omega_{l}^{p}\right)^{\frac{1}{p}}$
holds for $\omega_{1},\cdots,\omega_{l}\geq0$ and $p\in\left(0,1\right]$,
we can write 
\begin{align*}
\sum_{i\geq1}\norm{\vc d^{k}|_{\st T_{k,2i-1}}}_{p} & \leq\left(\sum_{i\geq1}\norm{\vc d^{k}|_{\st T_{k,2i-1}}}_{p}^{p}\right)^{\frac{1}{p}}.
\end{align*}
 The desired result then follows using the fact that the sets $\st T_{k,2i-1}$
are disjoint and $\bigcup_{i\geq1}\st T_{k,2i-1}=\st S\cmpl$.
\end{proof}
Proof of the following Lemma mostly relies on some common inequalities
that have been used in the compressed sensing literature (see e.g.,
\citep[Theorem 2.1]{CahrtrandICASSP07} and \citep[Theorem 2]{GribonvalNielsen2007})
.
\begin{lem}
\label{lem:ellpell2}The error vector $\vc d^{k}$ satisfies $\norm{\vc d^{k}|_{\st S\cmpl}}_{p}\leq s^{\frac{1}{p}-\frac{1}{2}}\norm{\vc d^{k}|_{\st S}}_{2}$
for all $k=0,1,2,\cdots$.\end{lem}
\begin{proof}
Since $\mbox{supp}\left(\vc x_{\perp}^{\star}\right)\subseteq\st S=\text{supp}\left(\vc x^{\star}\right)$
we have $\vc d^{k}|_{\st S\cmpl}=\vc x^{k}|_{\st S\cmpl}$. Furthermore,
because $\vc x^{k}$ is a feasible point by assumption we have $\norm{\vc x^{k}}_{p}^{p}\leq\widehat{c}=\norm{\vc x_{\perp}^{\star}}_{p}^{p}$
that implies, 
\begin{align*}
\norm{\vc d^{k}|_{\st S\cmpl}}_{p}^{p} & =\norm{\vc x^{k}|_{\st S\cmpl}}_{p}^{p}\\
 & \leq\norm{\vc x_{\perp}^{\star}}_{p}^{p}-\norm{\vc x^{k}|_{\st S}}_{p}^{p}\\
 & \leq\norm{\vc x_{\perp}^{\star}-\vc x^{k}|_{\st S}}_{p}^{p}\\
 & =\norm{\vc d^{k}|_{\st S}}_{p}^{p}\\
 & \leq s^{1-\frac{p}{2}}\norm{\vc d^{k}|_{\st S}}_{2}^{p}, & \mbox{(power means inequality})
\end{align*}
 which yields the desired result.
\end{proof}
The next lemma is a straightforward extension of a previously known
result \citep[Lemma 3.1]{RIP-OMP-Davenport} to the case of complex
vectors and asymmetric RIP.
\begin{lem}
\label{lem:GRIP}For $\vc u,\vc v\in\mathbb{C}^{n}$ suppose that
matrix $\mx A$ satisfies RIP of order $\max\left(\norm{\vc u+\vc v}_{0},\norm{\vc u-\vc v}_{0}\right)$
with constants $\alpha$ and $\beta$. Then we have 
\begin{align*}
\left|\Re\left[\eta\left\langle \mx A\vc u,\mx A\vc v\right\rangle -\left\langle \vc u,\vc v\right\rangle \right]\right| & \leq\left(\frac{\eta\left(\alpha-\beta\right)}{2}+\left|\frac{\eta\left(\alpha+\beta\right)}{2}-1\right|\right)\norm{\vc u}_{2}\norm{\vc v}_{2}.
\end{align*}
\end{lem}
\begin{proof}
If either of the vectors $\vc u$ and $\vc v$ is zero the claim becomes
trivial. So without loss of generality we assume that none of these
vectors is zero. The RIP condition holds for the vectors $\vc{u\pm v}$
and we have 
\begin{align*}
\beta\norm{\vc u\pm\vc v}_{2}^{2} & \leq\norm{\mx A\left(\vc u\pm\vc v\right)}_{2}^{2}\leq\alpha\norm{\vc u\pm\vc v}_{2}^{2}.
\end{align*}
 Therefore, we obtain 
\begin{align*}
\Re\left\langle \mx A\vc u,\mx A\vc v\right\rangle  & =\frac{1}{4}\left(\norm{\mx A\left(\vc u+\vc v\right)}_{2}^{2}-\norm{\mx A\left(\vc u-\vc v\right)}_{2}^{2}\right)\\
 & \leq\frac{1}{4}\left(\alpha\norm{\vc u+\vc v}_{2}^{2}-\beta\norm{\vc u-\vc v}_{2}^{2}\right)\\
 & =\frac{\alpha-\beta}{4}\left(\norm{\vc u}_{2}^{2}+\norm{\vc v}_{2}^{2}\right)+\frac{\alpha+\beta}{2}\Re\left\langle \vc u,\vc v\right\rangle .
\end{align*}
 Applying this inequality for vectors $\frac{\vc u}{\norm{\vc u}_{2}}$
and $\frac{\vc v}{\norm{\vc v}_{2}}$ yields 
\begin{align*}
\Re\left[\eta\left\langle \mx A\frac{\vc u}{\norm{\vc u}_{2}},\mx A\frac{\vc v}{\norm{\vc v}_{2}}\right\rangle -\left\langle \frac{\vc u}{\norm{\vc u}_{2}},\frac{\vc v}{\norm{\vc v}_{2}}\right\rangle \right] & \leq\frac{\eta\left(\alpha-\beta\right)}{2}+\left(\frac{\eta\left(\alpha+\beta\right)}{2}-1\right)\Re\left\langle \frac{\vc u}{\norm{\vc u}_{2}},\frac{\vc v}{\norm{\vc v}_{2}}\right\rangle \\
 & \leq\frac{\eta\left(\alpha-\beta\right)}{2}+\left|\frac{\eta\left(\alpha+\beta\right)}{2}-1\right|.
\end{align*}
 Similarly it can be shown that 
\begin{align*}
\Re\left[\eta\left\langle \mx A\frac{\vc u}{\norm{\vc u}_{2}},\mx A\frac{\vc v}{\norm{\vc v}_{2}}\right\rangle -\left\langle \frac{\vc u}{\norm{\vc u}_{2}},\frac{\vc v}{\norm{\vc v}_{2}}\right\rangle \right] & \geq-\frac{\eta\left(\alpha-\beta\right)}{2}-\left|\frac{\eta\left(\alpha+\beta\right)}{2}-1\right|.
\end{align*}
 The desired result follows immediately by multiplying the last two
inequalities by $\norm{\vc u}_{2}\norm{\vc v}_{2}$.\end{proof}
\begin{lem}
\label{lem:IterationInvariant}If the step-size of $\ell_{p}$-PGD
obeys $\left|\eta_{k}\left(\alpha_{3s}+\beta_{3s}\right)/2-1\right|\leq\tau$
for some $\tau\geq0$, then we have 
\begin{align*}
\Re\left[\left\langle \vc d^{k},\vc d^{k+1}\right\rangle -\eta_{k}\left\langle \mx A\vc d^{k},\mx A\vc d^{k+1}\right\rangle \right] & \leq\left(\left(1+\tau\right)\rho_{3s}+\tau\right)\left(1+\sqrt{2p}\left(\frac{2}{2-p}\right)^{\frac{1}{2}-\frac{1}{p}}\right)^{2}\norm{\vc d^{k}}_{2}\norm{\vc d^{k+1}}_{2}.
\end{align*}
\end{lem}
\begin{proof}
Note that 
\begin{align}
\Re\left[\left\langle \vc d^{k},\vc d^{k+1}\right\rangle -\eta_{k}\left\langle \mx A\vc d^{k},\mx A\vc d^{k+1}\right\rangle \right] & =\Re\left[\left\langle \vc d^{k}|_{\st T_{k}},\vc d^{k+1}|_{\st T_{k+1}}\right\rangle -\eta_{k}\left\langle \mx A\vc d^{k}|_{\st T_{k}},\mx A\vc d^{k+1}|_{\st T_{k+1}}\right\rangle \right]\nonumber \\
 & +\sum_{i\geq2}\Re\left[\left\langle \vc d^{k}|_{\st S_{k,i}},\vc d^{k+1}|_{\st T_{k+1}}\right\rangle -\eta_{k}\left\langle \mx A\vc d^{k}|_{\st S_{k,i}},\mx A\vc d^{k+1}|_{\st T_{k+1}}\right\rangle \right]\nonumber \\
 & +\sum_{j\geq2}\Re\left[\left\langle \vc d^{k}|_{\st T_{k}},\vc d^{k+1}|_{\st S_{k+1,j}}\right\rangle -\eta_{k}\left\langle \mx A\vc d^{k}|_{\st T_{k}},\mx A\vc d^{k+1}|_{\st S_{k+1,j}}\right\rangle \right]\nonumber \\
 & +\sum_{i,j\geq2}\Re\left[\left\langle \vc d^{k}|_{\st S_{k,i}},\vc d^{k+1}|_{\st S_{k+1,j}}\right\rangle -\eta_{k}\left\langle \mx A\vc d^{k}|_{\st S_{k,i}},\mx A\vc d^{k+1}|_{\st S_{k+1,j}}\right\rangle \right].\label{eq:L23E1}
\end{align}
Note that $\left|\st T_{k}\cup\st T_{k+1}\right|\leq3s$. Furthermore,
for $i,j\!\geq\!2$ we have $\left|\st T_{k}\cup\st S_{k+1,j}\right|\!\leq\!3s$,
$\left|\st T_{k+1}\cup\st S_{k,i}\right|\!\leq\!3s$, and $\left|\st S_{k,i}\cup\st S_{k+1,j}\right|\leq2s$.
Therefore, by applying Lemma \ref{lem:GRIP} for each of the summands
in \eqref{eq:L23E1} and using the fact that 
\begin{align*}
\rho'_{3s} & :=\left(1+\tau\right)\rho_{3s}+\tau\\
 & \geq\eta_{k}\left(\alpha_{3s}-\beta_{3s}\right)/2+\left|\eta_{k}\left(\alpha_{3s}+\beta_{3s}\right)/2-1\right|
\end{align*}
we obtain 
\begin{align*}
\Re\left[\left\langle \vc d^{k},\vc d^{k+1}\right\rangle -\eta_{k}\left\langle \mx A\vc d^{k},\mx A\vc d^{k+1}\right\rangle \right] & \leq\rho'_{3s}\norm{\vc d^{k}|_{\st T_{k}}}_{2}\norm{\vc d^{k+1}|_{\st T_{k+1}}}_{2}+\sum_{i\geq2}\rho'_{3s}\norm{\vc d^{k}|_{\st S_{k,i}}}_{2}\norm{\vc d^{k+1}|_{\st T_{k+1}}}_{2}\\
 & +\sum_{j\geq2}\rho'_{3s}\norm{\vc d^{k}|_{\st T_{k}}}_{2}\norm{\vc d^{k+1}|_{\st S_{k+1,j}}}_{2}+\sum_{i,j\geq2}\rho'_{3s}\norm{\vc d^{k}|_{\st S_{k,i}}}_{2}\norm{\vc d^{k+1}|_{\st S_{k+1,j}}}_{2}.
\end{align*}
 Hence, applying Lemma \ref{lem:ell2ellp} yields 
\begin{align*}
\Re\left[\left\langle \vc d^{k},\vc d^{k+1}\right\rangle -\eta_{k}\left\langle \mx A\vc d^{k},\mx A\vc d^{k+1}\right\rangle \right] & \leq\rho'_{3s}\norm{\vc d^{k}|_{\st T_{k}}}_{2}\norm{\vc d^{k+1}|_{\st T_{k+1}}}_{2}\\
 & +\sqrt{2p}\left(\frac{2s}{2-p}\right)^{\frac{1}{2}-\frac{1}{p}}\rho'_{3s}\norm{\vc d^{k}|_{\st S^{c}}}_{p}\norm{\vc d^{k+1}|_{\st T_{k+1}}}_{2}\displaybreak[0]\\
 & +\sqrt{2p}\left(\frac{2s}{2-p}\right)^{\frac{1}{2}-\frac{1}{p}}\rho'_{3s}\norm{\vc d^{k}|_{\st T_{k}}}_{2}\norm{\vc d^{k+1}|_{\st S^{c}}}_{p}\\
 & +2p\left(\frac{2s}{2-p}\right)^{1-\frac{2}{p}}\rho'_{3s}\norm{\vc d^{k}|_{\st S^{c}}}_{p}\norm{\vc d^{k+1}|_{\st S^{c}}}_{p}.
\end{align*}
Then it follows from Lemma \ref{lem:ellpell2}, 
\begin{align*}
\Re\left[\left\langle \vc d^{k},\vc d^{k+1}\right\rangle -\eta_{k}\left\langle \mx A\vc d^{k},\mx A\vc d^{k+1}\right\rangle \right] & \leq\rho'_{3s}\norm{\vc d^{k}|_{\st T_{k}}}_{2}\norm{\vc d^{k+1}|_{\st T_{k+1}}}_{2}\\
 & +\sqrt{2p}\left(\frac{2}{2-p}\right)^{\frac{1}{2}-\frac{1}{p}}\rho'_{3s}\norm{\vc d^{k}|_{\st S}}_{2}\norm{\vc d^{k+1}|_{\st T_{k+1}}}_{2}\\
 & +\sqrt{2p}\left(\frac{2}{2-p}\right)^{\frac{1}{2}-\frac{1}{p}}\rho'_{3s}\norm{\vc d^{k}|_{\st T_{k}}}_{2}\norm{\vc d^{k+1}|_{\st S}}_{2}\\
 & +2p\left(\frac{2}{2-p}\right)^{1-\frac{2}{p}}\rho'_{3s}\norm{\vc d^{k}|_{\st S}}_{2}\norm{\vc d^{k+1}|_{\st S}}_{2}\displaybreak[0]\\
 & \leq\rho'_{3s}\left(1+\sqrt{2p}\left(\frac{2}{2-p}\right)^{\frac{1}{2}-\frac{1}{p}}\right)^{2}\norm{\vc d^{k}}_{2}\norm{\vc d^{k+1}}_{2}
\end{align*}

\end{proof}
Now we are ready to prove the accuracy guarantees for the $\ell_{p}$-PGD
algorithm.
\begin{proof}[\textbf{Proof of Theorem \ref{thm:LpPGD}}]
 Recall that $\gamma$ is defined by \eqref{eq:gamma}. It follows
from Lemmas \ref{lem:OptimCondition} and \ref{lem:IterationInvariant}
that 
\begin{align*}
\norm{\vc d^{k}}_{2}^{2} & \leq2\gamma\norm{\vc d^{k}}_{2}\norm{\vc d^{k-1}}_{2}+2\eta_{k}\Re\left\langle \mx A\vc d^{k},\mx A\vc d^{\star}+\vc e\right\rangle \\
 & \leq2\gamma\norm{\vc d^{k}}_{2}\norm{\vc d^{k-1}}_{2}+2\eta_{k}\norm{\mx A\vc d^{k}}_{2}\norm{\mx A\vc d^{\star}+\vc e}_{2}.
\end{align*}
 Furthermore, using \eqref{eq:ell2ellpE1} and Lemma \ref{lem:ellpell2}
we deduce 
\begin{align*}
\norm{\mx A\vc d^{k}}_{2} & \leq\norm{\mx A\vc d^{k}|_{\st T_{k}}}_{2}+\sum_{i\geq1}\norm{\mx A\vc{\vc d}^{k}|_{\st T_{k,2i}}}_{2}\\
 & \leq\sqrt{\alpha_{2s}}\norm{\vc d^{k}|_{\st T_{k}}}_{2}+\sum_{i\geq1}\sqrt{\alpha_{2s}}\norm{\vc{\vc d}^{k}|_{\st T_{k,2i}}}_{2}\\
 & \leq\sqrt{\alpha_{2s}}\norm{\vc d^{k}|_{\st T_{k}}}_{2}+\sqrt{\alpha_{2s}}\sqrt{p}\left(\frac{2s}{2-p}\right)^{\frac{1}{2}-\frac{1}{p}}\sum_{i\geq1}\norm{\vc{\vc d}^{k}|_{\st T_{k,2i-1}}}_{p}\\
 & \leq\sqrt{\alpha_{2s}}\norm{\vc d^{k}|_{\st T_{k}}}_{2}+\sqrt{\alpha_{2s}}\sqrt{p}\left(\frac{2s}{2-p}\right)^{\frac{1}{2}-\frac{1}{p}}\norm{\vc{\vc d}^{k}|_{\st S^{c}}}_{p}\displaybreak[0]\\
 & \leq\sqrt{\alpha_{2s}}\norm{\vc d^{k}|_{\st T_{k}}}_{2}+\sqrt{\alpha_{2s}}\sqrt{p}\left(\frac{2}{2-p}\right)^{\frac{1}{2}-\frac{1}{p}}\norm{\vc{\vc d}^{k}|_{\st S}}_{2}\\
 & \leq\sqrt{\alpha_{2s}}\left(1+\sqrt{p}\left(\frac{2}{2-p}\right)^{\frac{1}{2}-\frac{1}{p}}\right)\norm{\vc{\vc d}^{k}}_{2}.
\end{align*}
Therefore, 
\begin{align*}
\norm{\vc d^{k}}_{2}^{2} & \leq2\gamma\norm{\vc d^{k}}_{2}\norm{\vc d^{k-1}}_{2}+2\eta_{k}\sqrt{\alpha_{2s}}\left(1+\sqrt{p}\left(\frac{2}{2-p}\right)^{\frac{1}{2}-\frac{1}{p}}\right)\norm{\vc d^{k}}_{2}\norm{\mx A\vc d^{\star}+\vc e}_{2},
\end{align*}
 which after canceling $\norm{\vc d^{k}}_{2}$ yields 
\begin{align*}
\norm{\vc d^{k}}_{2} & \leq2\gamma\norm{\vc d^{k-1}}_{2}+2\eta_{k}\sqrt{\alpha_{2s}}\left(1+\sqrt{p}\left(\frac{2}{2-p}\right)^{\frac{1}{2}-\frac{1}{p}}\right)\norm{\mx A\vc d^{\star}+\vc e}_{2}\\
 & =2\gamma\norm{\vc d^{k-1}}_{2}+2\eta_{k}\left(\alpha_{3s}+\beta_{3s}\right)\frac{\sqrt{\alpha_{2s}}}{\alpha_{3s}+\beta_{3s}}\left(1+\sqrt{p}\left(\frac{2}{2-p}\right)^{\frac{1}{2}-\frac{1}{p}}\right)\norm{\mx A\vc d^{\star}+\vc e}_{2}\\
 & \leq2\gamma\norm{\vc d^{k-1}}_{2}+4\left(1+\tau\right)\frac{\sqrt{\alpha_{2s}}}{\alpha_{3s}+\beta_{3s}}\left(1+\sqrt{p}\left(\frac{2}{2-p}\right)^{\frac{1}{2}-\frac{1}{p}}\right)\left(\norm{\mx A\vc d^{\star}}_{2}+\norm{\vc e}_{2}\right).
\end{align*}
Since $\vc x_{\perp}^{\star}$ is a projection of $\vc x^{\star}$
onto the feasible set $\widehat{\st B}$ and $\left(\frac{\widehat{c}}{\norm{\vc x^{\star}}_{p}^{p}}\right)^{1/p}\vc x^{\star}\in\widehat{\st B}$
we have 
\begin{align*}
\norm{\vc d^{\star}}_{2} & =\norm{\vc x_{\perp}^{\star}-\vc x^{\star}}_{2}\\
 & \leq\norm{\left(\frac{\widehat{c}}{\norm{\vc x^{\star}}_{p}^{p}}\right)^{1/p}\vc x^{\star}-\vc x^{\star}}_{2}\\
 & =\varepsilon\norm{\vc x^{\star}}_{2}.
\end{align*}
Furthermore, $\mbox{supp}\left(\vc d^{\star}\right)\subseteq\st S,$
thereby we can use RIP to obtain 
\begin{align*}
\norm{\mx A\vc d^{\star}}_{2} & \leq\sqrt{\alpha_{s}}\norm{\vc d^{\star}}_{2}\\
 & \leq\varepsilon\sqrt{\alpha_{s}}\norm{\vc x^{\star}}_{2}.
\end{align*}
Hence, 
\begin{align*}
\norm{\vc d^{k}}_{2} & \leq2\gamma\norm{\vc d^{k-1}}_{2}+4\left(1+\tau\right)\frac{\sqrt{\alpha_{2s}}}{\alpha_{3s}+\beta_{3s}}\left(1+\sqrt{p}\left(\frac{2}{2-p}\right)^{\frac{1}{2}-\frac{1}{p}}\right)\left(\varepsilon\sqrt{\alpha_{s}}\norm{\vc x^{\star}}_{2}+\norm{\vc e}_{2}\right)\\
 & \leq2\gamma\norm{\vc d^{k-1}}_{2}+2\left(1+\tau\right)\left(1+\sqrt{p}\left(\frac{2}{2-p}\right)^{\frac{1}{2}-\frac{1}{p}}\right)\left(\varepsilon\left(1+\rho_{3s}\right)\norm{\vc x^{\star}}_{2}+\frac{2\sqrt{\alpha_{2s}}}{\alpha_{3s}+\beta_{3s}}\norm{\vc e}_{2}\right).
\end{align*}
Applying this inequality recursively and using the fact that 
\begin{align*}
\sum_{i=0}^{k-1}\left(2\gamma\right)^{i} & <\sum_{i=0}^{\infty}\left(2\gamma\right)^{i}=\frac{1}{1-2\gamma},
\end{align*}
which holds because of the assumption $\gamma<\frac{1}{2}$, we can
finally deduce 
\begin{align*}
\norm{\vc x^{k}-\vc x^{\star}}_{2} & =\norm{\vc d^{k}-\vc d^{\star}}_{2}\\
 & \leq\norm{\vc d^{k}}_{2}+\norm{\vc d^{\star}}_{2}\\
 & \leq\left(2\gamma\right)^{k}\norm{\vc x_{\perp}^{\star}}_{2}+\frac{2\left(1+\tau\right)}{1-2\gamma}\left(1+\xi\left(p\right)\right)\left(\varepsilon\left(1+\rho_{3s}\right)\norm{\vc x^{\star}}_{2}+\frac{2\sqrt{\alpha_{2s}}}{\alpha_{3s}+\beta_{3s}}\norm{\vc e}_{2}\right)+\norm{\vc d^{\star}}_{2}\\
 & \leq\left(2\gamma\right)^{k}\norm{\vc x^{\star}}_{2}+\frac{2\left(1+\tau\right)}{1-2\gamma}\left(1+\xi\left(p\right)\right)\left(\varepsilon\left(1+\rho_{3s}\right)\norm{\vc x^{\star}}_{2}+\frac{2\sqrt{\alpha_{2s}}}{\alpha_{3s}+\beta_{3s}}\norm{\vc e}_{2}\right)+\varepsilon\norm{\vc x^{\star}}_{2},
\end{align*}
 where $\xi\left(p\right)=\sqrt{p}\left(\frac{2}{2-p}\right)^{\frac{1}{2}-\frac{1}{p}}$
as defined in the statement of the theorem.
\end{proof}

\section{\label{sec:discuss}Discussion}

In this paper we studied the accuracy of the Projected Gradient Descent
algorithm in solving sparse least squares problems where sparsity
is dictated by an $\ell_{p}$-norm constraint. Assuming that one has
an algorithm that can find a projection of any given point onto $\ell_{p}$-balls
with $p\in\left[0,1\right]$, we have shown that the PGD method converges
to the true signal, up to the statistical precision, at a linear rate.
The convergence guarantees in this paper are obtained by requiring
proper RIP conditions to hold for the measurement matrix. By varying
$p$ from zero to one, these sufficient conditions become more stringent
while robustness to noise and convergence rate worsen. This behavior
suggests that smaller values of $p$ are preferable, and in fact the
PGD method at $p=0$ (i.e., the IHT algorithm) outperforms the PGD
method at $p>0$ in every aspect. These conclusions, however, are
not definitive as we have merely presented sufficient conditions for
accuracy of the PGD method.

Unfortunately and surprisingly, for $p\in\left(0,1\right)$ the algorithm
for projection onto $\ell_{p}$-balls is not as simple as the cases
of $p=0$ and $p=1$, leaving practicality of the algorithm unclear
for the intermediate values $p$. We have shown (see the Appendix)
that a projection $\vc x^{\perp}$ of point $\vc x\in\mathbb{C}^{n}$
has the following properties
\begin{enumerate}[label=\textnormal{(\roman*)}]
\item  $\left|x_{i}^{\perp}\right|\leq\left|x_{i}\right|$ for all $i\in\left[n\right]$
while there is at most one $i\in\left[n\right]$ such that $\left|x_{i}^{\perp}\right|<\frac{1-p}{2-p}\left|x_{i}\right|$,
\item $\Arg\left(x_{i}\right)=\Arg\left(x_{i}^{\perp}\right)$ for $i\in\left[n\right]$,
\item if $\left|x_{i}\right|>\left|x_{j}\right|$ for some $i,j\in\left[n\right]$
then $\left|x_{i}^{\perp}\right|\geq\left|x_{j}^{\perp}\right|$,
and
\item there exist $\lambda\geq0$ such that for all $i\in\text{supp}\left(\vc x^{\perp}\right)$
we have $\left|x_{i}^{\perp}\right|^{1-p}\left(\left|x_{i}\right|-\left|x_{i}^{\perp}\right|\right)=p\lambda$. 
\end{enumerate}
However, these properties are not sufficient for full characterization
of a projection. One may ask that if the PGD method performs the best
at $p=0$ then why is it important at all to design a projection algorithm
for $p>0$?  We believe that developing an efficient algorithm for
projection onto $\ell_{p}$-balls with $p\in\left(0,1\right)$ is
an interesting problem that can provide a building block for other
methods of sparse signal estimation involving the $\ell_{p}$-norm.
Furthermore, studying this problem may help to find an insight on
how the complexity of these algorithms vary in terms of $p$.

In future work, we would like to examine the performance of more sophisticated
first-order methods such as the Nesterov's optimal gradient methods
\citep{Nesterov04} for $\ell_{p}$-constrained least squares problems.
Furthermore, it could be possible to extend the provided framework
further to analyze $\ell_{p}$-constrained minimization with objective
functions other than the squared error. This generalized framework
can be used in problems such as regression with generalized linear
models that arise in statistics and machine learning.


\bibliographystyle{elsart-num-sort}
\bibliography{references}

\appendix
\renewcommand{\thesection}{Appendix \Alph{section}}

\section{Lemmas for Characterization of a Projection onto $\ell_{p}$-balls}

\renewcommand{\thesection}{\Alph{section}}

In what follows we assume that $\st B$ is an $\ell_{p}$-ball with
$p$-radius $c$ (i.e., $\st B=\st F_{p}\left(c\right)$). For $\vc x\in\mathbb{C}^{n}$
we derive some properties of 
\begin{align}
\vc x^{\perp} & \in\arg\min\ \frac{1}{2}\norm{\vc x-\vc u}_{2}^{2}\quad\mbox{s.t. }\vc u\in\st B,\label{eq:LpProjection}
\end{align}
a projection of $\vc x$ onto $\st B$.
\begin{lem}
\label{lem:PhaseMatch}Let $\vc x^{\perp}$ be a projection of $\vc x$
onto $\st B$. Then for every $i\in\left\{ 1,2,\ldots,n\right\} $
we have $\mathrm{Arg}\left(x_{i}\right)=\mathrm{Arg}\left(x_{i}^{\perp}\right)$
and $\left|x_{i}^{\perp}\right|\leq\left|x_{i}\right|$.\end{lem}
\begin{proof}
Proof by contradiction. Suppose that for some $i$ we have $\mathrm{Arg}\left(x_{i}\right)\neq\mathrm{Arg}\left(x_{i}^{\perp}\right)$
or $\left|x_{i}^{\perp}\right|>\left|x_{i}\right|$. Consider the
vector $\vc x'$ for which $x'_{j}=x_{j}^{\perp}$ for $j\neq i$
and $x'_{i}=\min\left\{ \left|x_{i}\right|,\left|x_{i}^{\perp}\right|\right\} \exp\left(\mbox{\ensuremath{\imath}}\mathrm{Arg}\left(x_{i}\right)\right)$
(the character $\imath$ denotes the imaginary unit $\sqrt{-1}$).
We have $\norm{\vc x'}_{p}\leq\norm{\vc x^{\perp}}_{p}$ which implies
that $\vc x'\in\st B$. Since $\left|x_{i}-x'_{i}\right|<\left|x_{i}-x_{i}^{\perp}\right|$
we have $\norm{\vc x'-\vc x}_{2}<\norm{\vc x^{\perp}-\vc x}_{2}$
which contradicts the choice of $\vc x^{\perp}$ as a projection.\end{proof}
\begin{assumption*}
Lemma \ref{lem:PhaseMatch} asserts that the projection $\vc x^{\perp}$
has the same phase components as $\vc x$. Therefore, without loss
of generality and for simplicity in the following lemmas we assume
$\vc x$ has real-valued non-negative entries.\end{assumption*}
\begin{lem}
\label{lem:Sorted}For any $\vc x$ in the positive orthant there
is a projection $\vc x^{\perp}$ of $\vc x$ onto the set $\st B$
such that for $i,j\in\left\{ 1,2,\ldots,n\right\} $ we have $x_{i}^{\perp}\leq x_{j}^{\perp}$
iff $x_{i}\leq x_{j}$.\end{lem}
\begin{proof}
Note that the set $\st B$ is closed under any permutation of coordinates.
In particular, by interchanging the $i$-th and $j$-th entries of
$\vc x^{\perp}$ we obtain another vector $\vc x'$ in $\st B$. Since
$\vc x^{\perp}$ is a projection of $\vc x$ onto $\st B$ we must
have $\left\Vert \vc x-\vc x^{\perp}\right\Vert _{2}^{2}\leq\left\Vert \vc x-\vc x'\right\Vert _{2}^{2}$.
Therefore, we have $\left(x_{i}-x_{i}^{\perp}\right)^{2}+\left(x_{j}-x_{j}^{\perp}\right)^{2}\leq\left(x_{i}-x_{j}^{\perp}\right)^{2}+\left(x_{j}-x_{i}^{\perp}\right)^{2}$
and from that $0\leq\left(x_{i}-x_{j}\right)\left(x_{i}^{\perp}-x_{j}^{\perp}\right).$
For $x_{i}\neq x_{j}$ the result follows immediately, and for $x_{i}=x_{j}$
without loss of generality we can assume $x_{i}^{\perp}\leq x_{j}^{\perp}$. \end{proof}
\begin{lem}
\label{lem:LagrangeOptimal}Let $\st S^{\perp}$ be the support set
of $\vc x^{\perp}$. Then there exists a $\lambda\geq0$ such that
\begin{align*}
x_{i}^{\perp(1-p)}\left(x_{i}-x_{i}^{\perp}\right) & =p\lambda
\end{align*}
 for all $i\in\st S^{\perp}$.\end{lem}
\begin{proof}
The fact that $\vc x^{\perp}$ is a solution to the minimization expressed
in \eqref{eq:LpProjection} implies that that $\vc x^{\perp}|_{\st S^{\perp}}$
must be a solution to 
\begin{align*}
\arg\min_{\vc v} & \ \frac{1}{2}\left\Vert \vc x|_{\st S^{\perp}}-\vc v\right\Vert _{2}^{2}\quad\text{s.t. }\left\Vert \vc v\right\Vert _{p}^{p}\leq c.
\end{align*}
The normal to the feasible set (i.e., the gradient of the constraint
function) is uniquely defined at $\vc x^{\perp}|_{\st S^{\perp}}$
since all of its entries are positive by assumption. Consequently,
the Lagrangian 
\begin{align*}
L\left(\vc v,\lambda\right) & =\frac{1}{2}\left\Vert \vc x|_{\st S^{\perp}}-\vc v\right\Vert _{2}^{2}\!+\lambda\left(\left\Vert \vc v\right\Vert _{p}^{p}-c\right)
\end{align*}
 has a well-defined partial derivative $\frac{\partial L}{\partial\vc v}$
at $\vc x^{\perp}|_{\st S^{\perp}}$ which must be equal to zero for
an appropriate $\lambda\geq0$. Hence, 
\begin{align*}
\forall i & \in\st S^{\perp}\ x_{i}^{\perp}-x_{i}+p\lambda x_{i}^{\perp\left(p-1\right)}=0
\end{align*}
 which is equivalent to the desired result.\end{proof}
\begin{lem}
\label{lem:ShrinkageCurves}Let $\lambda\geq0$ and $p\in\left[0,1\right]$
be fixed numbers and set $T_{0}=\left(2-p\right)\left(p\left(1-p\right)^{p-1}\lambda\right)^{\frac{1}{2-p}}$.
Denote the function $t^{1-p}\left(T-t\right)$ by $h_{p}\left(t\right)$.
The following statements hold regarding the roots of $h_{p}\left(t\right)=p\lambda$:
\begin{enumerate}[label=\textnormal{(\roman*)}]
\item  \label{enu:ShrinkagrCurvesPart1}For $p=1$ and $T\geq T_{0}$ the
equation $h_{1}\left(t\right)=\lambda$ has a unique solution at $t=T-\lambda\in\left[0,T\right]$
which is an increasing function of $T$.
\item \label{enu:ShrinkagrCurvesPart2}For $p\in\left[0,1\right)$ and $T\geq T_{0}$
the equation $h_{p}\left(t\right)=p\lambda$ has two roots $t_{-}$
and $t_{+}$ satisfying $t_{-}\in\left(0,\frac{1-p}{2-p}T\right]$
and $t_{+}\in\left[\frac{1-p}{2-p}T,+\infty\right)$. As a function
of $T$, $t_{-}$ and $t_{+}$ are decreasing and increasing, respectively
and they coincide at $T=T_{0}$.
\end{enumerate}
\end{lem}
\begin{proof}
Fig. \ref{fig:ShrinkageCurves} illustrates $h_{p}\left(t\right)$
for different values of $p\in\left[0,1\right]$. To verify part \ref{enu:ShrinkagrCurvesPart1}
observe that we have $T_{0}=\lambda$ thereby $T\geq\lambda$. The
claim is then obvious since $h_{1}\left(t\right)-\lambda=T-t-\lambda$
is zero at $t=T-\lambda$. Part \ref{enu:ShrinkagrCurvesPart2} is
more intricate and we divide it into two cases: $p=0$ and $p\neq0$.
At $p=0$ we have $T_{0}=0$ and $h_{0}\left(t\right)=t\left(T-t\right)$
has two zeros at $t_{-}=0$ and $t_{+}=T$ that obviously satisfy
the claim. So we can now focus on the case $p\in\left(0,1\right)$.
It is straightforward to verify that $t_{\max}=\frac{1-p}{2-p}T$
is the location at which $h_{p}\left(t\right)$ peaks. Straightforward
algebraic manipulations also show that $T>T_{0}$ is equivalent to
$p\lambda<h_{p}\left(t_{\max}\right)$. Furthermore, inspecting the
sign of $h'_{p}\left(t\right)$ shows that $h_{p}\left(t\right)$
is strictly increasing over $\left[0,t_{\max}\right]$ while it is
strictly decreasing over $\left[t_{\max},T\right]$. Then, using the
fact that $h_{p}\left(0\right)=h_{p}\left(T\right)=0\leq p\lambda<h_{p}\left(t_{\max}\right)$,
it follows from the \emph{intermediate value theorem} that $h_{p}\left(t\right)=p\lambda$
has exactly two roots, $t_{-}$ and $t_{+}$, that straddle $t_{\max}$
as claimed. Furthermore, taking the derivative of $t_{-}^{1-p}\left(T-t_{-}\right)=p\lambda$
with respect to $T$ yields 
\begin{align*}
\left(1-p\right)t_{-}'t_{-}^{-p}\left(T-t_{-}\right)+t_{-}^{1-p}\left(1-t_{-}'\right) & =0.
\end{align*}
 Hence, 
\begin{align*}
\left(\left(1-p\right)\left(T-t_{-}\right)-t_{-}\right)t_{-}' & =-t_{-}
\end{align*}
which because $t_{-}\leq t_{\max}=\frac{1-p}{2-p}T$ implies that
$t_{-}'<0$. Thus $t{}_{-}$ is a decreasing function of $T$. Similarly
we can show that $t_{+}$ is an increasing function of $T$ using
the fact that $t_{+}\geq t_{\max}$. Finally, as $T$ decreases to
$T_{0}$ the peak value $h_{p}\left(t_{\max}\right)$ decreases to
$p\lambda$ which implies that $t_{-}$ and $t_{+}$ both tend to
the same value of $\frac{1-p}{2-p}T_{0}$.
\end{proof}
\noindent \begin{center}
\begin{figure}
\noindent \centering\includegraphics[width=0.65\textwidth]{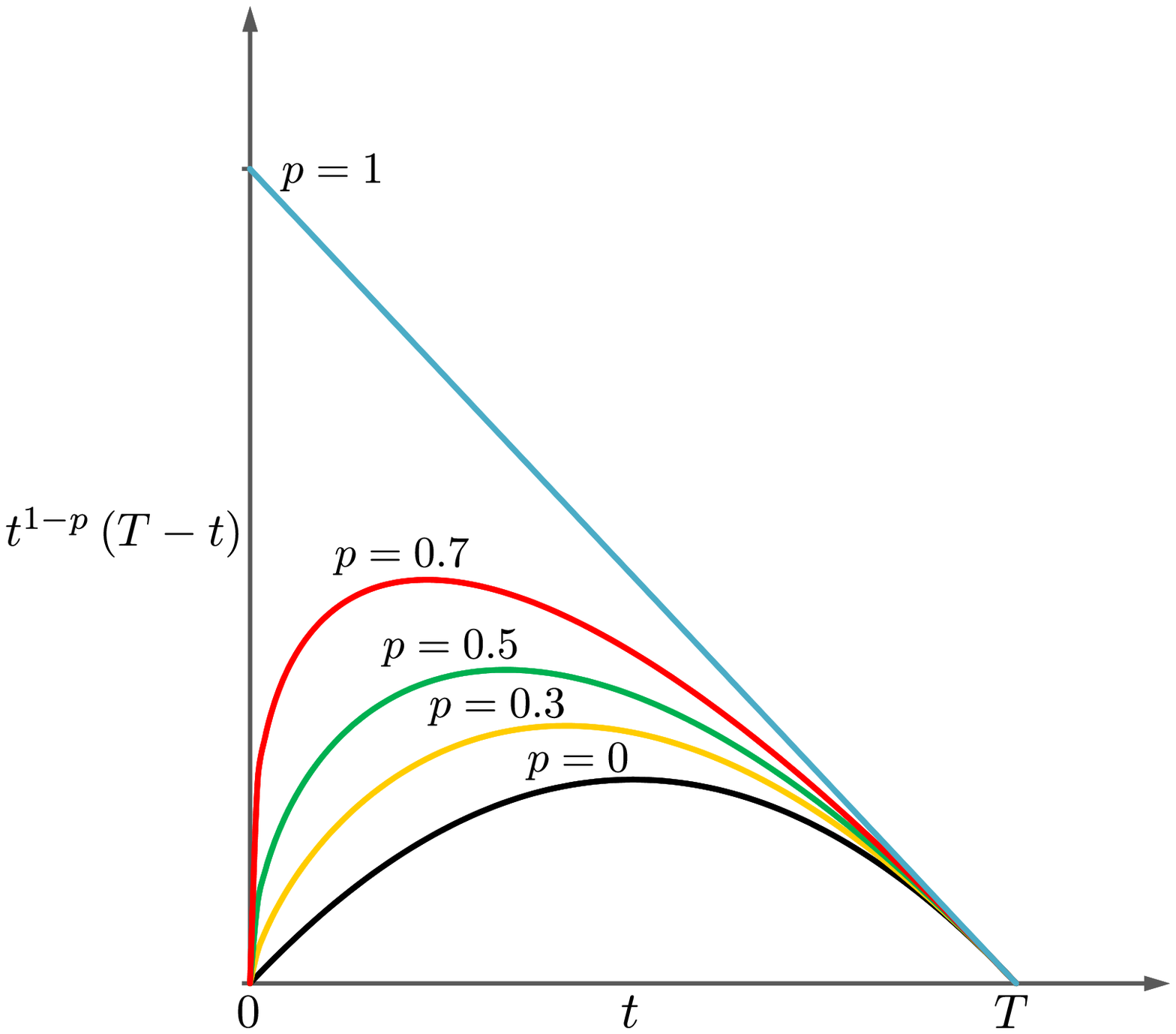}

\caption{\label{fig:ShrinkageCurves}The function $t^{1-p}\left(T-t\right)$
for different values of $p$}
\end{figure}

\par\end{center}
\begin{lem}
\label{lem:EqualMinEntries}Suppose that $x_{i}=x_{j}>0$ for some
$i\neq j$. If $x_{i}^{\perp}=x_{j}^{\perp}>0$ then $x_{i}^{\perp}\geq\frac{1-p}{2-p}x_{i}$
.\end{lem}
\begin{proof}
For $p\in\left\{ 0,1\right\} $ the claim is obvious since at $p=0$
we have $x_{i}^{\perp}=x_{i}>\frac{1}{2}x_{i}$ and at $p=1$ we have
$\frac{1-p}{2-p}x_{i}=0$. Therefore, without loss of generality we
assume $p\in\left(0,1\right)$. The proof is by contradiction. Suppose
that $w=\frac{x_{i}^{\perp}}{x_{i}}=\frac{x_{j}^{\perp}}{x_{j}}<\frac{1-p}{2-p}$.
Since $\vc x^{\perp}$ is a projection it follows that $a=b=w$ must
be the solution to 
\begin{align*}
\arg\min_{a,b}\ \psi=\frac{1}{2}\left[\left(1-a\right)^{2}+\left(1-b\right)^{2}\right]\ \text{ s.t. }a^{p}+b^{p} & =2w^{p},\: a>0,\text{ and }b>0,
\end{align*}
otherwise the vector $\vc x'$ that is identical to $\vc x^{\perp}$
except for $x'_{i}=ax_{i}\neq x_{i}^{\perp}$ and $x'_{j}=bx_{j}\neq x_{i}^{\perp}$
is also a feasible point (i.e., $\vc x'\in\st B$) that satisfies
\begin{align*}
\norm{\vc x'-\vc x}_{2}^{2}-\norm{\vc x^{\perp}-\vc x}_{2}^{2} & =\left(1-a\right)^{2}x_{i}^{2}+\left(1-b\right)^{2}x_{j}^{2}-\left(1-w\right)^{2}x_{i}^{2}-\left(1-w\right)^{2}x_{j}^{2}\\
 & =\left(\left(1-a\right)^{2}+\left(1-b\right)^{2}-2\left(1-w\right)^{2}\right)x_{i}^{2}\\
 & <0,
\end{align*}
 which is absurd. If $b$ is considered as a function of $a$ then
$\psi$ can be seen merely as a function of $a$, i.e., $\psi\equiv\psi\left(a\right)$.
Taking the derivative of $\psi$ with respect to $a$ yields 
\begin{align*}
\psi'\left(a\right) & =a-1+b'\left(b-1\right)\\
 & =a-1-\left(\frac{a}{b}\right)^{p-1}\left(b-1\right)\\
 & =\left(b^{1-p}\left(1-b\right)-a^{1-p}\left(1-a\right)\right)a^{p-1}\\
 & =\left(2-p\right)(b-a)\nu^{-p}\left(\frac{1-p}{2-p}-\nu\right),
\end{align*}
where the last equation holds by the \emph{mean value theorem} for
some $\nu\in\left(\min\left\{ a,b\right\} ,\max\left\{ a,b\right\} \right)$.
Since $w<\frac{1-p}{2-p}$ we have $r_{1}:=\min\left\{ 2^{1/p}w,\frac{1-p}{2-p}\right\} >w$
and $r_{0}:=\left(2w^{p}-r_{1}^{p}\right)^{1/p}<w$. With straightforward
algebra one can show that if either $a$ or $b$ belongs to the interval
$\left[r_{0},r_{1}\right]$, then so does the other one. By varying
$a$ in $\left[r_{0},r_{1}\right]$ we always have $\nu<r_{1}\leq\frac{1-p}{2-p}$,
therefore as $a$ increases in this interval the sign of $\psi'$
changes at $a=w$ from positive to negative. Thus, $a=b=w$ is a local
maximum of $\psi$ which is a contradiction.\end{proof}

\end{document}